\documentclass[a4pper,10.5pt]{article}
\usepackage{geometry}

\usepackage{amsthm}
\usepackage{amsmath}
\usepackage{mathtools}
\usepackage{mathdots}
\usepackage{color}
\newtheorem{theorem}{Theorem}[section]

\allowdisplaybreaks[0]
\theoremstyle{definition}
\newtheorem{definition}[theorem]{Definition}
\newtheorem{example}[theorem]{Example}

\usepackage{enumerate}
\usepackage{multirow}
\usepackage{setspace}
\usepackage{algpseudocode}
\usepackage{algorithm}
\usepackage{lipsum}
\usepackage{amsfonts}
\usepackage{graphicx}
\usepackage[figuresright]{rotating}
\usepackage{float}
\usepackage{subfigure}
\usepackage{authblk}
\usepackage{epstopdf}
\allowdisplaybreaks

\theoremstyle{remark}
\newtheorem{remark}[theorem]{Remark}

\numberwithin{equation}{section}

\newcommand{\ds}{\mathrm ds}

\newcommand{\red}{\color{red}}

\makeatletter
\renewcommand*\env@matrix[1][\arraystretch]{%
  \edef\arraystretch{#1}%
  \hskip -\arraycolsep
  \let\@ifnextchar\new@ifnextchar
  \array{*\c@MaxMatrixCols c}}
\makeatother

\begin{document}



\title{Roundoff error problem in L2-type methods for time-fractional problems}          


\author{Chaoyu Quan \thanks{ SUSTech International Center for Mathematics, Southern University of Science and Technology, Shenzhen, China,\quad
Guangdong Provincial Key Laboratory of Computational Science and Material Design, Southern University of Science and Technology, Shenzhen, China (quanchaoyu@gmail.com).}
\quad Shijie Wang \thanks{Department of Mathematics, Southern University of Science and Technology, Shenzhen, China, (pariwong@163.com).}  \quad Xu Wu \thanks{Department of Mathematics,  Harbin Institute of Technology, Harbin, 150001,China,  \quad Department of Mathematics, Southern University of Science and Technology, Shenzhen, China (xwsustech@gmail.com).}}


  \date{}

\maketitle

\begin{abstract}
Roundoff error problems have occurred frequently in interpolation methods of time-fractional equations, which can lead to undesirable results such as the failure of optimal convergence. 
These problems are essentially caused by catastrophic cancellations.
Currently, a feasible way to avoid these cancellations is using the Gauss--Kronrod quadrature to approximate the integral formulas of coefficients rather than computing the explicit formulas directly for example in the L2-type methods. 
This nevertheless increases computational cost and arises additional integration errors.
In this work, a new framework to handle catastrophic cancellations is proposed, in particular, in the computation of the coefficients for standard and fast L2-type methods on general nonuniform meshes.
We propose a concept of $\delta$-cancellation and then some threshold conditions ensuring that $\delta$-cancellations will not happen. 
If the threshold conditions are not satisfied, a Taylor-expansion technique is proposed to avoid $\delta$-cancellation. 
Numerical experiments show that our proposed method performs as accurate as the Gauss--Kronrod quadrature method and meanwhile much more efficient. 
This enables us to complete long time simulations with hundreds of thousands of time steps in short time. 
\end{abstract}



\section{Introduction}
The Caputo derivative \cite{caputo1967linear} has been widely-used to model phenomena which takes account of interactions within the past and problems with nonlocal properties. It is defined as
\begin{equation*}
    \partial_t^\alpha u \coloneqq \frac{1}{\Gamma(1-\alpha)} \int_0^t \frac{u'(s)}{(t-s)^\alpha} \, \ds.
\end{equation*}
One issue is how to compute the Caputo derivative numerically, which is important in numerical simulations of time-fractional problems.
Some approximations have been developed and analyzed, including the piecewise polynomial interpolation methods (such as L1, L2-1$_\sigma$ and L2 schemes) on uniform meshes \cite{langlands2005accuracy,sun2006fully, alikhanov2015new, gao2014new, lv2016error, quan2022energy} and nonuniform meshes \cite{stynes2017error, kopteva2019error, chen2019error, kopteva2020error, kopteva2021error, liao2018sharp, liao2019discrete,liao2018second}, discontinuous Galerkin methods \cite{mustapha2014discontinuous}, and convolution quadrature correction methods \cite{jin2017correction,jin2020subdiffusion}.
The solutions to time-fractional problems typically admit weak singularities, which leads to deterioration of convergence in the case of uniform meshes for interpolation methods and then inspires researchers to use nonuniform meshes. 
Another important issue is the CPU time and storage problem due to the nonlocality of the fractional derivatives.
To reduce the computation cost, fast algorithms are proposed, for example, using the sum-of-exponentials (SOE) approximation to the Caputo derivative \cite{li2021second, liu2022unconditionally,ke2015fast,zhu2019fast}.

Despite of many works on the convergence analysis, the aforementioned interpolation methods could encounter roundoff error problems in practice where the accuracy is destroyed as well as the convergence rate. 
However, only a few literatures have discussed about this issue.
For example, in \cite{liao2018second,chen2019error}, the authors use the Gauss--Kronrod quadrature to compute the coefficients in the second-order L2-1$_\sigma$ method for time-fractional diffusion problem.
Even earlier, in \cite{jiang2017fast}, Jiang et al. encounter the cancellation error in the fast L1 method and use the Taylor expansion of exponentials with a small number of terms to overcome it. 
Despite of some attempts, a systematic analysis of the roundoff error problem in the interpolation methods remains a gap. 



In this work, we focus on solving the roundoff error problems in standard and fast L2 methods on general nonuniform time meshes.
We consider particularly the L2-type method proposed in \cite{lv2016error} but on general nonuniform meshes. 
When calculating the explicit expressions of L2 coefficients, catastrophic cancellations could happen, the phenomenon that subtracting good approximations to two nearby numbers may yield a very bad approximation to the difference of the original numbers. 
This can lead to completely wrong simulations especially for a large number of time steps.
To quantify the phenomena of catastrophic cancellation, we first propose a concept of $\delta-$cancellation, where $\delta$ denotes the tolerance of relative error of the approximated difference (due to rounding) and the exact difference.
Then we reformulate the standard L2 coefficients into a combination of $I_1,~I_2$ in \eqref{eq:stdI} and  the fast L2 coefficients into a combination of $J_1,~J_2$ in \eqref{eq:soeJ} to make the analysis of $\delta$-cancellation simpler.  
For any fixed $\alpha\in (0,1)$ and $\delta > 12\delta_0$, we define the following thresholds
    \begin{equation*}
    \begin{aligned}
           \theta_{s,1}=\frac{2\delta_0}{(1-\alpha)\delta},\ \theta_{s,2}=\left(\frac{6\delta_0}{(1-\alpha)\delta}\right)^{\frac12},
              \ \theta_{f,1}&=\frac{4\delta_0}{\delta},\ \theta_{f,2}=\left(\frac{12\delta_0}{\delta}\right)^\frac12,
    \end{aligned}
    \end{equation*}
where $\delta_0$ is the machine error ($\delta_0\approx 2.22\times 10^{-16}$ with 64 bits in double precision).
We show that for any $j$ and $k$, if the following threshold condition
\begin{equation*}
    \tau_j/(t_k - t_{j-1}) \ge\theta_{s,i} 
\end{equation*}
holds where $\tau_j$ is the $j$th time step and $t_k$ is the $k$th time node, then $I_i$ will not meet $\delta$-cancellation ($i=1,2$). 
Similarly,  for any $\ell$ and $k$, if the following threshold condition
\begin{equation*}
    \theta^{\ell}\tau_{k-1} \ge\theta_{f,i}
\end{equation*}
holds where $\theta^\ell$ is the $\ell$th quadrature node in the SOE approximation, then $J_i$ will not meet $\delta$-cancellation ($i=1,2$). 
In other words, if the above threshold conditions are satisfied, one can compute the L2 coefficients directly from the explicit expressions (see Theorem \ref{lemma1}). 
If the threshold conditions are not satisfied, we propose to use Taylor's expansion to approximate $I_i$ and $J_i$ to avoid $\delta$-cancellation,  called the Taylor-expansion technique in this work.
Usually, only several terms are needed to ensure the relative error of Taylor approximation within machine error.
Numerical experiments show that this method using the threshold conditions plus the Taylor-expansion technique, called TCTE method, can perform as accurate as the Gauss--Kronrod quadrature, but much more efficient.
This enables us to implement long time simulations with hundreds of thousands of time steps.

This work is organized as follows. 
In Section~\ref{sect2}, we describe the roundoff error problems in the standard and fast L2 methods. 
In Section~\ref{sect3}, we propose the concept of $\delta-$cancellation and provide the threshold conditions when $\delta-$cancellation will not happen, and then proposed the Taylor-expansion technique to avoid $\delta-$cancellation if the threshold conditions are not satisfied.
In Section~\ref{sect4}, some numerical experiments are given to verify the efficiency and accuracy of the proposed TCTE method. 
Some conclusions are provided in the final section.

\section{Catastrophic cancellations in standard and fast  L2 formulas}\label{sect2}
In this part, we describe the catastrophic cancellations encountered in the L2-type methods, which can lead to bad results when computing the explicit formulas of L2 coefficients.
Note that such problems can also happen in L1 and L2-1$_\sigma$ methods.

\subsection{Reformulation of standard L2 formula}
Given a general nonuniform time mesh $t_0<t_1<\ldots<t_k$, the Caputo fractional derivative at time $t_k$ is approximated by the L2 discrete fractional operator (see \cite{quan2022h} for details)
\begin{equation}\label{eq:oriL2}
\begin{aligned}
L_1^\alpha u& = \frac{u^1-u^0}{\Gamma(2-\alpha)\tau_1^{\alpha}},\\
    L_k^\alpha u 
    & = \frac{1}{\Gamma(1-\alpha)}\left( \sum_{j=1}^{k-1} \left(a_{j}^{(k)} u^{j-1} + b_{j}^{(k)} u^j + c_{j}^{(k)} u^{j+1} \right) 
    + a^{(k)}_k u^{k-2} + b^{(k)}_k  u^{k-1} + c^{(k)}_k u^{k}\right), \quad k\geq 2,
\end{aligned}   
\end{equation}
where $ b_{j}^{(k)}= -a_{j}^{(k)}- c_{j}^{(k)}$ for $1\leq j\leq k$,  and for $1\leq j\leq k-1$,
\begin{align}\label{eq:akj}
   a_{j}^{(k)} & =  \int_{t_{j-1}}^{t_j} \frac{2s -t_j-t_{j+1}}{\tau_{j}(\tau_{j}+\tau_{j+1})} \frac{1}{(t_k-s)^\alpha}\,\ds 
  ,\quad
    c_{j}^{(k)}  =  \int_{t_{j-1}}^{t_j} \frac{2s -t_{j-1}-t_{j}}{\tau_{j+1}(\tau_{j}+\tau_{j+1})} \frac{1}{(t_k-s)^\alpha}\,\ds,
\end{align}
and
\begin{equation}\label{akkckk}
\begin{aligned}
    a^{(k)}_{k}  =  \int_{t_{k-1}}^{t_k} \frac{2s -t_{k-1}-t_{k}}{\tau_{k-1}(\tau_{k-1}+\tau_{k})} \frac{1}{(t_k-s)^\alpha}\,\ds 
 ,\quad
   c^{(k)}_{k}  =  \int_{t_{k-1}}^{t_k} \frac{2s -t_{k-2}-t_{k-1}}{\tau_{k}(\tau_{k-1}+\tau_{k})} \frac{1}{(t_k-s)^\alpha}\,\ds.
\end{aligned}
\end{equation}Note that we have the following relationship:
\begin{equation}\label{eq:tilde}
     c_j^{(k)}  
   = \frac{\tau_{j}}{\tau_{j+1}}a_j^{(k)}+ \tilde c_j^{(k)} \quad \text{with }\quad \tilde c_j^{(k)} \coloneqq \int_{t_{j-1}}^{t_j}\frac{1}{\tau_{j+1}(t_k-s)^\alpha}\, \ds.
\end{equation}
Then we can reformulate \eqref{eq:oriL2} for $k\ge 2$ as
\begin{equation}\label{eq:refL2}
\begin{aligned}
    L_k^\alpha u 
    & = \frac{1}{\Gamma(1-\alpha)}\left( \sum_{j=1}^{k-1} \left(a_{j}^{(k)}\left( {\tau_j}/{\tau_{j+1}}\delta_{j+1}u-\delta_ju\right) + \tilde c_{j}^{(k)}\delta_{j+1} u \right) 
    - a^{(k)}_k \delta_{k-1}u+ c^{(k)}_k \delta_{k}u\right), 
\end{aligned}   
\end{equation}
where $\delta_ju\coloneqq u^j-u^{j-1}$.
Note that the reformulation \eqref{eq:refL2} will help to obtain simpler threshold conditions in later analysis. 
We can figure out the explicit expression of $a_j^{(k)}$, $\tilde c_j^{(k)}$,  in \eqref{eq:akj} and \eqref{eq:tilde},  for 
$1\leq j\leq k-1$:
\begin{equation}\label{eq:akjexplicitit}
\begin{aligned}
   a_j^{(k)}  =-\frac{ (2-\alpha)\tau_{j+1}I_1+2 I_2}{(2-\alpha)(1-\alpha)\tau_j(\tau_j+\tau_{j+1})},\quad 
    \tilde c_j^{(k)}  =
   \frac{I_1}{(1-\alpha)\tau_{j+1}}
\end{aligned}
\end{equation}
with 
\begin{equation}\label{eq:stdI}
    \begin{aligned}
        I_1 =  (t_k - t_{j-1})^{1-\alpha} - (t_k - t_j )^{1-\alpha},\quad
        I_2 =  (2-\alpha) \tau_j (t_k - t_{j-1})^{1-\alpha} + (t_k - t_j)^{2 -\alpha} - (t_k - t_{j-1})^{2-\alpha},
         \end{aligned}
\end{equation}
and \begin{equation*}\label{eq:a_k1}
\begin{aligned}
    a_k^{(k)}
    =\frac{\alpha\tau_k^2}{(2-\alpha)(1-\alpha)\tau_{k-1}(\tau_{k-1}+\tau_{k}) \tau_k^\alpha},\quad c_k^{(k)}
    =\frac{1}{(1-\alpha)\tau_k^\alpha}+\frac{\alpha\tau_k}{(2-\alpha)(1-\alpha)(\tau_{k-1}+\tau_{k}) \tau_k^\alpha}.
    \end{aligned}
\end{equation*}
It is not difficult to verify that $I_1>0$ and $I_2>0$.

\subsection{Reformulation of fast L2 formula} 
We consider the fast L2 formula  obtained by applying the sum-of-exponentials approximation \cite{jiang2017fast,liao2019unconditional} to the historical part of the standard L2 formula. 
Specifically speaking, the fast L2 formula reads
\begin{align*}\label{eq:soel2}
F_1^\alpha u= \frac{u^1-u^0}{\Gamma(2-\alpha)\tau_1^{\alpha}},\quad
    F_k^\alpha u
     = \frac{1}{\Gamma(1-\alpha)}\left( 
  \sum_{\ell=1}^{N_{q}}\varpi^{\ell} H^\ell(t_{k})
    -a^{(k)}_k \delta_{k-1}u + c^{(k)}_k\delta_k  u \right), \quad k\geq 2,
\end{align*}   
where  $a^{(k)}_k,\ c^{(k)}_k$ are given in \eqref{akkckk}, and 
\begin{equation}\label{eq:recH}
\begin{aligned}
     H^\ell(t_{k})=e^{-\theta^{\ell}\tau_k}H^\ell(t_{k-1})+ {a}^{k,\ell}_{k-1}u^{k-2}+ {b}^{k,\ell}_{k-1}u^{k-1}+ {c}^{k,\ell}_{k-1}u^{k},
\end{aligned}
\end{equation}
where $\theta^{\ell}$ and $\varpi^{\ell}$ are positive quadrature nodes and weights in the SOE approximation, $N_q$ is the number of  quadrature nodes, $H^\ell(t_{1})=0 $, $ {b}^{k,\ell}_{k-1}= - {a}^{k,\ell}_{k-1} - {c}^{k,\ell}_{k-1} $ and
\begin{align}\label{eq:fastakj}
    {a}^{k,\ell}_{k-1} =  \int_{t_{k-2}}^{t_{k-1}} \frac{2s -t_{k-1}-t_{k}}{\tau_{k-1}(\tau_{k-1}+\tau_{k})} e^{-\theta^{\ell}(t_k-s)}\,\ds ,\quad
   {c}^{k,\ell}_{k-1} =  \int_{t_{k-2}}^{t_{k-1}} \frac{2s -t_{k-2}-t_{k-1}}{\tau_{k}(\tau_{k-1}+\tau_{k})}e^{-\theta^{\ell}(t_k-s)}\,\ds .
\end{align}
Note that
\begin{equation}\label{eq:fastckj}
   {c}^{k,\ell}_{k-1}  =\frac{\tau_{k-1}}{\tau_k}{a}^{k,\ell}_{k-1}+ \tilde{c}^{k,\ell}_{k-1}\quad \text{with}\quad  \tilde{c}^{k,\ell}_{k-1} \coloneqq\int_{t_{k-2}}^{t_{k-1}}\frac{e^{-\theta^{\ell}(t_k-s)}}{\tau_k}\,\ds.
\end{equation}
We can reformulate \eqref{eq:recH} as 
\begin{equation}\label{eq:comHk}
       H^\ell(t_{k})=e^{-\theta^{\ell}\tau_k}H^\ell(t_{k-1})+ {a}^{k,\ell}_{k-1}\left(\tau_{k-1}/\tau_k\delta_k u-\delta_{k-1}u\right)+ \tilde{c}^{k,\ell}_{k-1}\delta_k u.
\end{equation}
We can figure out the explicit expression of   $ {a}^{k,\ell}_{k-1}$, $  \tilde{c}^{k,\ell}_{k-1}$  in  \eqref{eq:fastakj} and \eqref{eq:fastckj}:
\begin{equation}\label{eq:soeexplic}
\begin{aligned}
   {a}^{k,\ell}_{k-1}=-\frac{e^{-\theta^{\ell}\tau_k}(\theta^{\ell}\tau_kJ_1+2J_2)}{\tau_{k-1}(\tau_{k-1}+\tau_{k})(\theta^{\ell})^2},\quad
   \tilde{c}^{k,\ell}_{k-1}=\frac{e^{-\theta^{\ell}\tau_k}J_1}{\theta^{\ell}\tau_{k}}
\end{aligned}
\end{equation}
with 
\begin{equation}\label{eq:soeJ}
    \begin{aligned}
     J_1=1-e^{-\theta^{\ell}\tau_{k-1}},\quad
     J_2=1-\theta^{\ell}\tau_{k-1}e^{-\theta^{\ell}\tau_{k-1}}-e^{-\theta^{\ell}\tau_{k-1}}.
    \end{aligned}
\end{equation}
It is not difficult to verify that $J_1>0$ and $J_2>0$.
\subsection{Catastrophic cancellations in $I_1,~I_2,~J_1$ and $J_2$}

We introduce a bit about the catastrophic cancellation for readers. 
Consider the subtraction of two numbers $x$ and $y$.
Assume that the approximations (probably caused by rounding in floating point arithmetic) of $x$ and $y$ are
\begin{equation*} 
\tilde {x}=x(1+\delta _{x}),\qquad \tilde {y}=y(1+\delta _{y}),
\end{equation*}
where $|\delta _{x}|=|x-{\tilde {x}}|/|x|$ and $|\delta _{y}|=|y-{\tilde {y}}|/|y|$ are relative errors. 
Then the relative error of the approximate difference $\tilde {x}-\tilde {y}$ from the true difference $x-y$ is inversely proportional to the true difference:
\begin{equation*}
    \begin{aligned}
{\tilde {x}}-{\tilde {y}}=x(1+\delta _{x})-y(1+\delta _{y})=(x-y)\left(1+{\frac {x\delta _{x}-y\delta _{y}}{x-y}}\right).
\end{aligned}
        \end{equation*}
Thus, the relative error of $ {\tilde {x}}-{\tilde {y}}$ and $x-y$ is
\begin{equation}\label{eq:relerr}
     \left|\frac {(\tilde x-\tilde y) - (x-y)}{x-y}\right| = \left|\frac {x\delta _{x}-y\delta _{y}}{x-y}\right|,
\end{equation}
which can be arbitrarily large if the true inputs $x$ and $y$ are close.

Catastrophic cancellation happens because subtraction is ill-conditioned at nearby inputs.
One can find that $I_1$ and $I_2$ in \eqref{eq:akjexplicitit} can come across the catastrophic cancellation when $\tau_j\ll t_k-t_{j-1}$, while  $J_1$ and  $J_2$ in \eqref{eq:soeexplic} can come across it when  $\theta^{\ell}\tau_{k-1}\ll 1$.
To avoid the catastrophic cancellation, the adaptive Gauss–Kronrod quadrature can be used to approximate the integral formula of standard and fast L2-1$_\sigma$ coefficients (similar to the L2 coefficients), rather than compute the explicit formula directly \cite{liao2018second,chen2019error,quan2022stability,quanwuyang2022h}. 
The integration quadrature is to compute all coefficients which will certainly increase the computation cost.
However, this is unnecessary in our opinion because the catastrophic cancellation does not usually happen.  

Next, we provide a theory on when the catastrophic cancellation might happen.
If the catastrophic cancellation does not happen, there is no doubt that computing the explicit expressions \eqref{eq:akjexplicitit} and \eqref{eq:soeexplic} directly is the most efficient way. 
If the catastrophic cancellation might happen, we propose to use Taylor-expansion approximation to do the calculation.

\section{Threshold conditions and Taylor-expansion approximation}\label{sect3}
In this part, we shall propose a new method to deal with the roundoff error problems in standard and fast L2 methods.


In fact, $I_1$ and $I_2$ of \eqref{eq:stdI} can be rewritten as 
\begin{equation}\label{eq:I12}
\begin{aligned}
I_1  =(t_k - t_{j-1})^{1-\alpha} \left[1-(1-\theta)^{1-\alpha} \right], \quad
I_2  =  (t_k - t_{j-1})^{2-\alpha} \left[(2-\alpha) \theta +  (1 - \theta)^{2-\alpha} - 1\right],
\end{aligned}
\end{equation}
where 
\begin{equation}\label{eq:theta}
    0<\theta = \frac{\tau_j}{t_k - t_{j-1}}<1,\quad 1\leq j\leq k-1.
\end{equation}
In \eqref{eq:I12}, it is clear that the catastrophic cancellations happen only when $\theta$ is very small.
Using the Taylor's expansions of $(1-\theta)^{1-\alpha}$ and $(1-\theta)^{2-\alpha}$, 
we have
\begin{equation}\label{eq:taylorIo}
\begin{aligned}
I_1 =- (t_k - t_{j-1})^{1-\alpha} \sum_{m = 1}^\infty \binom{1-\alpha}{m} (-\theta)^m, \quad 
I_2  = (t_k - t_{j-1})^{2-\alpha} \sum_{m = 2}^\infty \binom{2-\alpha}{m} (-\theta)^m.
\end{aligned}
\end{equation}
It is clear that
$
    \binom{1-\alpha}{m} (-\theta)^m
$
is always negative for any $m\geq 1$, and
$
    \binom{2-\alpha}{m} (-\theta)^m
$
is always positive for any $m\geq 2$. 
As a consequence, $I_1>0$, $I_2>0$, and the catastrophic cancellation will not happen in the Taylor expansion \eqref{eq:taylorIo}.
This indicates that $I_1$ and $I_2$ can be approximated by the following truncated Taylor's expansions (to avoid the catastrophic cancellation):
\begin{equation}\label{eq:taylorI}
    \begin{aligned}
 \hat I_1 = -(t_k - t_{j-1})^{1-\alpha} \sum_{m = 1}^{M_1} \binom{1-\alpha}{m} (-\theta)^m, \quad
 \hat I_2 =  (t_k - t_{j-1})^{2-\alpha} \sum_{m = 2}^{M_2+1} \binom{2-\alpha}{m} (-\theta)^m,
    \end{aligned}
\end{equation}
where $M_1\geq 1$ and $M_2\geq 1$ are the truncation numbers respectively for $I_1$ and $I_2$.


Similarly, applying the Taylor's expansions of $e^{\theta^{\ell}\tau_{k-1}}$  to $J_1$ and $J_2$ in \eqref{eq:soeexplic}, we have 
\begin{equation}\label{eq:taylorJo}
    \begin{aligned}
     J_1=e^{-\theta^{\ell}\tau_{k-1}}\sum_{m=1}^\infty\frac{(\theta^{\ell}\tau_{k-1})^m}{m!},\quad 
     J_2=e^{-\theta^{\ell}\tau_{k-1}}\sum_{m=2}^\infty\frac{(\theta^{\ell}\tau_{k-1})^m}{m!}.
    \end{aligned}
\end{equation}
Then we have $J_1>0$, $J_2>0$, and the catastrophic cancellation will not happen in the Taylor expansions \eqref{eq:taylorJo}.
Therefore, $J_1$ and $J_2$ can be approximated by the following truncated Taylor's expansions (to avoid the catastrophic cancellation  when $\theta^{\ell}\tau_{k-1}$ is very small):
\begin{equation}\label{eq:taylorJ}
\begin{aligned}
    \hat  J_1&=e^{-\theta^{\ell}\tau_{k-1}}\sum_{m=1}^{N_1}\frac{(\theta^{\ell}\tau_{k-1})^m}{m!},\quad
    \hat J_2 =e^{-\theta^{\ell}\tau_{k-1}}\sum_{m=2}^{N_2+1}\frac{(\theta^{\ell}\tau_{k-1})^m}{m!},
      \end{aligned}
\end{equation}
where $N_1\geq 1$ and $N_2\geq 1$ are the truncation numbers respectively for $J_1$ and $J_2$. 


We have shown that the truncated Taylor expansions \eqref{eq:taylorI} and \eqref{eq:taylorJ} can be used to avoid the catastrophic cancellation for the computation of $I_i$ and $J_i$, $i=1,~2$.
However, in most cases, the catastrophic cancellation won't happen and direct computation can already provide accurate results. 
To quantify the catastrophic cancellation phenomenon, we first introduce a concept of $\delta$-cancellation.
\begin{definition}[$\delta$-cancellation] 
       Given any two numbers $x$ and $y$, and their approximations $\tilde x$ and $\tilde y$, if the relative error  
   \begin{equation*}
    \left|\frac {(\tilde x-\tilde y) - (x-y)}{x-y}\right| \ge \delta,
\end{equation*}
then we call the subtraction $\tilde x-\tilde y$ a $\delta$-cancellation.
\end{definition}

According to \eqref{eq:relerr}, it is clear that if
\begin{equation*}\label{cancellation}
 \frac {|x\delta _{x}|+|y\delta _{y}|}{|x-y|}< \delta,
\end{equation*} 
the $\delta$-cancellation will never happen.
Next we state and prove the following theorem on the threshold conditions when $\delta$-cancellation won't happen.  
\begin{theorem}\label{lemma1}
    Given the fractional order $\alpha\in (0,1)$, the machine relative error $\delta_0$, and the parameter $ \delta >12\delta_0$ in $\delta$-cancellation, let
    \begin{equation*}\label{eq:lemma}
    \begin{aligned}
           \theta_{s,1}=\frac{2\delta_0}{(1-\alpha)\delta},\ \theta_{s,2}=\left(\frac{6\delta_0}{(1-\alpha)\delta}\right)^{\frac12},
              \ \theta_{f,1}&=\frac{4\delta_0}{\delta},\ \theta_{f,2}=\left(\frac{12\delta_0}{\delta}\right)^\frac12.
    \end{aligned}
    \end{equation*}
 If the threshold condition
 \begin{equation}\label{eq:tc1}
 \tau_j/(t_k - t_{j-1}) \ge\theta_{s,i}
 \end{equation}
 holds, then the subtraction in $I_i$ in \eqref{eq:stdI} is not $\delta$-cancellation for $i=1,~2$. 
 Similarly, if the threshold condition
 \begin{equation}\label{eq:tc2}
 \theta^{\ell}\tau_{k-1} \ge\theta_{f,i}
 \end{equation}
 holds,  then the subtraction in $J_i$ in  \eqref{eq:soeJ} is not $\delta$-cancellation for $i=1,~2$. 
\end{theorem}

\begin{proof}
 {\em $I_1$ case}: 
 Let $x_1=1$ and $y_1=(1-\theta)^{1-\alpha}$ be the two terms in $I_1$ where $\theta = \tau_j/(t_k - t_{j-1})<1$.
Since the relative errors $\delta _{x_1}$ and $\delta _{y_1}$ of float-point numbers $\tilde x_1$ and $\tilde y_1$ stored in machine are less than $\delta_0$, we have
\begin{equation*}
 \frac {|x_1\delta _{x_1}|+|y_1\delta _{y_1}|}{|x_1-y_1|}\le \frac {|x_1|+|y_1|}{|x_1-y_1|} \delta_0< \frac{\delta_0\left(1+(1-\theta)^{1-\alpha}\right)}{(1-\alpha)\theta}< \frac{2\delta_0}{(1-\alpha)\theta_{s,1}}=\delta,\quad \mbox{when }\theta\ge \theta_{s,1}.
\end{equation*}

{\em $I_2$ case: }
 Let $x_2=(2-\alpha) \theta +  (1 - \theta)^{2-\alpha}$ and $y_2=1$ be the two terms in $I_2$.
Since the relative errors $\delta _{x_2}$ and $\delta _{y_2}$ of float-point numbers $\tilde x_2$ and $\tilde y_2$ stored in machine are less than $\delta_0$, we have
\begin{equation*}
\begin{aligned}
  \frac {|x_2\delta _{x_2}|+|y_1\delta _{y_2}|}{|x_2-y_2|}\le \frac {|x_2|+|y_2|}{|x_2-y_2|} \delta_0 &\le2\frac {(2-\alpha) \theta +  (1 - \theta)^{2-\alpha}+1}{(1-\alpha)\theta^2} \delta_0 <
 \frac{6\delta_0}{(1-\alpha)\theta_{s,2}^2}=\delta,\quad \mbox{when }  \theta\ge \theta_{s,2}.
\end{aligned}
\end{equation*}
Here we use the fact that $(2-\alpha) \theta +  (1 - \theta)^{2-\alpha}$ increases w.r.t. $\theta\in (0,1)$.

 {\em $J_1$ case:} Let $x_3=1$ and $y_3=e^{-\theta^{\ell}\tau_{k-1}}$ be the two terms in $J_1$. 
 Since the relative errors $\delta _{x_3}$ and $\delta _{y_3}$ of float-point numbers $\tilde x_3$ and $\tilde y_3$ stored in machine are less than $\delta_0$, we have
 \begin{equation*}
  \frac {|x_3\delta _{x_3}|+|y_3\delta _{y_3}|}{|x_3-y_3|}\le \frac {|x_3|+|y_3|}{|x_3-y_3|} \delta_0< \frac{2\delta_0}{1-e^{-\theta_{f,1}}}< \frac{4\delta_0}{\theta_{f,1}}=\delta,\quad \mbox{when } \theta^{\ell}\tau_{k-1}\ge\theta_{f,1}.
\end{equation*}
Here we use the facts that $\theta_{f,1}<1$ and $1-e^{-\theta}>\theta/2$ for $\theta\in (0,1)$.


{\em $J_2$ case:} 
 Let $x_4=1$ and $y_4=\theta^{\ell}\tau_{k-1}e^{-\theta^{\ell}\tau_{k-1}}+e^{-\theta^{\ell}\tau_{k-1}}$ be the two terms in $J_2$.
Since the relative errors $\delta _{x_4}$ and $\delta _{y_4}$ of float-point numbers $\tilde x_4$ and $\tilde y_4$ stored in machine are less than $\delta_0$, we have
    \begin{equation*}
  \frac {|x_4\delta _{x_4}|+|y_4\delta _{y_4}|}{|x_4-y_4|}\le \frac {|x_4|+|y_4|}{|x_4-y_4|} \delta_0< \frac{2\delta_0}{1-e^{-\theta_{f,2}}-\theta_{f,2}e^{-\theta_{f,2}}} <\frac{12\delta_0}{\theta_{f,2}^2}=\delta,\quad \mbox{when } \theta^{\ell}\tau_{k-1}\ge\theta_{f,2}.
\end{equation*}
Here we use the facts that $\theta_{f,2}<1$ and $1-e^{-\theta}-\theta e^{-\theta}>\theta^2/6$ for $\theta\in (0,1)$.
\end{proof}

If the threshold condition \eqref{eq:tc1} or \eqref{eq:tc2} is not satisfied, the truncated Taylor expansion formulas, i.e. $\hat I_i$ and $\hat J_i$, can approximate $I_i$ and $J_i$ properly. 
\begin{theorem}\label{lemma2}
If $\tau_j/(t_k - t_{j-1}) \le \theta_{s,i}$ and  $\theta^{\ell}\tau_{k-1} \le \theta_{f,i}$ for $i=1,~2$, the  relative error $\delta I_i$ of $I_i$ in \eqref{eq:stdI} and $\hat I_i$ in \eqref{eq:taylorI}, and the relative error $\delta J_i$ of  $J_i$ in \eqref{eq:soeJ} and $\hat J_i$ in \eqref{eq:taylorJ} satisfy
  \begin{equation*}
      \begin{aligned}
     |\delta I_i| < \theta^{M_i},\quad |\delta J_i| <(\theta^{\ell}\tau_{k-1})^{N_i}
   ,\qquad i = 1,~2,
 \end{aligned}
 \end{equation*}
 where $\theta = \tau_j/(t_k - t_{j-1})$.
\end{theorem}
\begin{proof} From \eqref{eq:taylorIo}--\eqref{eq:taylorJ}, we have 
{\small
 \begin{equation*}
 \begin{aligned}
     | I_1-\hat I_1| &=\bigg |- (t_k - t_{j-1})^{1-\alpha} \sum_{m = M_1+1}^\infty \binom{1-\alpha}{m} (-\theta)^m\bigg| =\theta^{M_1}\bigg |- (t_k - t_{j-1})^{1-\alpha} \sum_{m = 1}^\infty \binom{1-\alpha}{m+M_1} (-\theta)^m\bigg|\\
      &<\theta^{M_1}\bigg |- (t_k - t_{j-1})^{1-\alpha} \sum_{m = 1}^\infty \binom{1-\alpha}{m} (-\theta)^m\bigg|= \theta^{M_1} I_1,\\
 | I_2-\hat I_2| &=\bigg |(t_k - t_{j-1})^{2-\alpha} \sum_{m = M_2+2}^\infty \binom{2-\alpha}{m} (-\theta)^m\bigg | = \theta^{M_2} \bigg | (t_k - t_{j-1})^{2-\alpha} \sum_{m = 2}^\infty \binom{2-\alpha}{m+M_2} (-\theta)^m\bigg | \\
&<\theta^{M_2} \bigg | (t_k - t_{j-1})^{2-\alpha} \sum_{m = 2}^\infty \binom{2-\alpha}{m} (-\theta)^m\bigg | = \theta^{M_2} I_2,\\
   | J_1-\hat J_1|& = e^{-\theta^{\ell}\tau_{k-1}}\sum_{m=N_1+1}^\infty\frac{(\theta^{\ell}\tau_{k-1})^m}{m!} = (\theta^{\ell}\tau_{k-1})^{N_1} e^{-\theta^{\ell}\tau_{k-1}}\sum_{m=1}^\infty\frac{(\theta^{\ell}\tau_{k-1})^m}{(m+N_1)!}\\&< (\theta^{\ell}\tau_{k-1})^{N_1}e^{-\theta^{\ell}\tau_{k-1}}\sum_{m=1}^\infty\frac{(\theta^{\ell}\tau_{k-1})^m}{m!} = (\theta^{\ell}\tau_{k-1})^{N_1} J_1,\\
    | J_2-\hat J_2| &= e^{-\theta^{\ell}\tau_{k-1}}\sum_{m=N_2+2}^\infty\frac{(\theta^{\ell}\tau_{k-1})^m}{m!} = (\theta^{\ell}\tau_{k-1})^{N_2}e^{-\theta^{\ell}\tau_{k-1}}\sum_{m=2}^\infty\frac{(\theta^{\ell}\tau_{k-1})^m}{(m+N_2)!}\\& < (\theta^{\ell}\tau_{k-1})^{N_2} e^{-\theta^{\ell}\tau_{k-1}}\sum_{m=2}^\infty\frac{(\theta^{\ell}\tau_{k-1})^m}{m!} = (\theta^{\ell}\tau_{k-1})^{N_2}J_2.
 \end{aligned}
 \end{equation*}}
 Then the  relative errors satisfy
 \begin{equation*}
      \begin{aligned}
     | \delta I_i| &<\theta^{M_i},\quad| \delta J_i| 
    <(\theta^{\ell}\tau_{k-1})^{N_i}, \qquad  i=1,~2. 
 \end{aligned}
 \end{equation*}
\end{proof}

We can select 
\begin{equation*}
    M_i=\lceil \log_\theta \delta_0 \rceil,\quad N_i=\lceil \log_{\theta^{\ell}\tau_{k-1}} \delta_0 \rceil,
\end{equation*}  
to ensure $\theta^{M_i}\le\delta_0$ and $(\theta^{\ell}\tau_{k-1})^{N_i}\le \delta_0$. 
According to Theorem \ref{lemma2},  $\hat I_i$ has the same floating-point number with $I_i$ in machine, while $\hat J_i$ has another same floating-point number with $J_i$.

Combining Theorem~\ref{lemma1}--\ref{lemma2}, we conclude that if the threshold conditions \eqref{eq:tc1} and \eqref{eq:tc2} are satisfied, $\delta$-cancellation won't happen in $I_i$ and $J_i$; if \eqref{eq:tc1} and \eqref{eq:tc2} are not satisfied, the Taylor expansions $\hat I_i$ and $\hat J_i$ can be applied to compute $I_i$ and $J_i$ (usually with only several terms).
We call this method the threshold conditions plus Taylor expansions (TCTE) method in this work.
At the end of this part, we illustrate our TCTE methods in Algorithms~\ref{alg:akjckj}--\ref{alg:fastack} of computing  standard L2 coefficients $a_j^{(k)}$, $\tilde c_j^{(k)}$, for $1\le j\le k-1$, and fast L2 coefficients  ${a}^{k,\ell}_{k-1}$, $\tilde{c}^{k,\ell}_{k-1}$, in \eqref{eq:akjexplicitit} and \eqref{eq:soeexplic}. 
  \begin{remark}
     For the  standard and fast L2-1$_\sigma$ methods, the correspond quantities $I_i$ and $J_i$ are computed by  replacing $t_k$ with $t_k-\alpha/2\tau_k$ and {\red $e^{-\theta^{\ell}\tau_k}$ with $e^{-\theta^{\ell}(1-\alpha/2)\tau_k}$, } in Algorithms~\ref{alg:akjckj}--\ref{alg:fastack}.
 \end{remark}

 \begin{algorithm}
  \setstretch{0.05}
 \caption{Compute $a_j^{(k)}$, $\tilde c_j^{(k)}$ in equation \eqref{eq:akjexplicitit} for L2 method.}
	\label{alg:akjckj}
 \raggedright
	Input the parameters $\alpha$, $\theta_{s,1}$, $\theta_{s,2}$,   $\tau_j$, $\tau_{j+1}$,  $t_j$  and  $t_k$.
	
   Compute
	          \begin{equation*}
		        \theta=\frac{\tau_j}{t_k-t_j+\tau_j}.
		      \end{equation*}
		      
		      	\textbf{If} $  \theta\le\theta_{s,1} $
         
                 \begin{align*}
                  &M_1=\lceil \log_\theta \delta_0 \rceil,\\
              & I_1= -(t_k-t_j+\tau_j)^{1-\alpha} \sum_{m = 1}^{M_1}\binom{1-\alpha}{m} (-\theta)^m,
                  \end{align*}
                  
                  	\textbf{else}
                  	\begin{equation*}
                  	   I_1  =(t_k-t_j+\tau_j)^{1-\alpha} \left[1-(1-\theta)^{1-\alpha} \right].
                  	\end{equation*}
                  	
                 	\textbf{end}
                  		\ 
                    
               \textbf{If} $  \theta\le\theta_{s,2} $
               
             \begin{align*}
                 &M_2=\lceil\log_\theta \delta_0\rceil,\\
                 &I_2= (t_k-t_j+\tau_j)^{2-\alpha} \sum_{m = 2}^{M_2+1} \binom{2-\alpha}{m} (-\theta)^m,
                  \end{align*}
                  
                 \textbf{else}
                  	\begin{equation*}
                  	  I_2 =  (t_k-t_j+\tau_j)^{2-\alpha} \left[(2-\alpha) \theta +  (1 - \theta)^{2-\alpha} - 1\right].
                  	\end{equation*}
                  	
                  \textbf{end}
                
   Compute and output     
\begin{equation*}
           a_j^{(k)}  =-\frac{ (2-\alpha)\tau_{j+1}I_1+2 I_2}{(2-\alpha)(1-\alpha)\tau_j(\tau_j+\tau_{j+1})},\quad 
        \tilde c_j^{(k)}  =
      \frac{I_1}{(1-\alpha)\tau_{j+1}}.
    \end{equation*}
\end{algorithm}
\begin{algorithm}
\setstretch{0.05}
 \caption{Compute ${a}^{k,\ell}_{k-1}$, $\tilde {c}^{k,\ell}_{k-1}$ in equation \eqref{eq:soeexplic} for fast L2 method. }
	\label{alg:fastack}
  \raggedright
	Input the parameters $\theta_{f,1}$, $\theta_{f,2}$,   $\theta^{\ell}$, $\tau_{k-1}$ and $\tau_{k}$.
	
	\textbf{If} $ \theta^{\ell}\tau_{k-1}\le\theta_{f,1} $

                 \begin{align*}
                 &N_1=\lceil \log_{\theta^{\ell}\tau_{k-1}} \delta_0 \rceil,\\
              &J_1=e^{-\theta^{\ell}\tau_{k-1}}\sum_{m=1}^{N_1}\frac{(\theta^{\ell}\tau_{k-1})^m}{m!},
                  \end{align*}
                  
                  	\textbf{else}
                  	\begin{equation*}
                  	    J_1=1-e^{-\theta^{\ell}\tau_{k-1}}.
                  	\end{equation*}
                  	
                 	\textbf{end}
                  		
               \textbf{If} $ \theta^{\ell}\tau_{k-1}\le\theta_{f,2} $

                 \begin{align*}
                 &N_2=\lceil \log_{\theta^{\ell}\tau_{k-1}} \delta_0 \rceil,\\
                   &J_2 =e^{-\theta^{\ell}\tau_{k-1}}\sum_{m=2}^{N_2+1}\frac{(\theta^{\ell}\tau_{k-1})^m}{m!},
                  \end{align*}
                  
                 \textbf{else}
                  	\begin{equation*}
                   J_2=1-\theta^{\ell}\tau_{k-1}e^{-\theta^{\ell}\tau_{k-1}}-e^{-\theta^{\ell}\tau_{k-1}}.
                  	\end{equation*}
                  	
                  \textbf{end}
                  
Compute and output
\begin{equation*}
       {a}^{k,\ell}_{k-1}=-\frac{e^{-\theta^{\ell}\tau_k}(\theta^{\ell}\tau_kJ_1+2J_2)}{\tau_{k-1}(\tau_{k-1}+\tau_{k})(\theta^{\ell})^2},\quad
   \tilde{c}^{k,\ell}_{k-1}=\frac{e^{-\theta^{\ell}\tau_k}J_1}{\theta^{\ell}\tau_{k}}.
\end{equation*}

\end{algorithm}

\section{Numerical experiments}\label{sect4}
In this section, we give some numerical tests to illustrate the accuracy and efficiency of our TCTE method. 
All experiments are implemented on a computer with 3.60GHz, Intel-Core i9-9900K  in Matlab  with 64 bits in double precision, where the machine error is $\delta_0 = 2^{-52}\approx$ 2.22e-16.

In the following content, we set  by default $\theta_{s,1}=\theta_{f,1}=10^{-4}$,  $\theta_{s,2}=\theta_{f,2}=10^{-2}$  in Algorithm~\ref{alg:akjckj}--\ref{alg:fastack}.
Based  on Theorem~\ref{lemma1}--\ref{lemma2}, the relative errors of $I_1,~I_2,~J_1,~J_2$ in the computation of L2 coefficients using the TCTE method satisfy 
 \begin{equation*}
     |\delta I_1|< \frac{4.44\text{e-12}}{1-\alpha},\quad |\delta I_2|<  \frac{1.332\text{e-11}}{1-\alpha},\quad |\delta J_1|<  8.88\text{e-12},\quad |\delta J_2|<  2.664\text{e-11}.
 \end{equation*}

We focus on the following linear subdiffusion equation:
\begin{equation}\label{eq:subdiffusion}
\begin{aligned}
    \partial_t^\alpha u(t,x) & = \Delta u(t,x)+f(t,x),&& (t,x)\in (0,T]\times\Omega,\\
    u(t,x) &= 0,&& (t,x)\in (0,T]\times\partial\Omega,\\
    u(0,x)& = u^0(x),&& x\in \Omega
\end{aligned}    
\end{equation}
by using the standard and fast L2 implicit schemes. 
The graded mesh \cite{stynes2017error} with grading parameter $r\geq 1$ is used
\begin{equation*}\label{eq:tauj}
\begin{aligned}
    t_j  = \left(\frac{j}{N}\right)^r  T, \quad\tau_j  = t_j -t_{j-1} = \left[\left(\frac {j} { N}\right)^r-\left(\frac {j-1}{ N}\right)^r\right]  T,
\end{aligned}
\end{equation*}
where $N$ is the number of time steps. 

\subsection{Accuracy of TCTE method} 
\begin{example}\label{exm1}
Consider the subdiffusion equation
  \eqref{eq:subdiffusion}  in $\Omega=[-1,1]^2$ with $T=1$ and
$f(t,x,y)=(\Gamma(1+\alpha)+2\pi^2t^\alpha)\sin(\pi x)\sin(\pi y)$, whose exact solution is $u(t,x,y)=t^\alpha\sin(\pi x)\sin(\pi y)$.
\end{example} 

In this example, the spectral collocation method \cite{trefethen2000spectral,shen2011spectral} is applied in space with $20^2$ Chebyshev--Gauss--Lobatto points.  We set the final time of SOE $T_{\rm soe}=T$, the SOE tolerance $\varepsilon=1\text{e-12}$, and the cut-off time $\Delta t= \tau_2$; see the SOE approximation \cite{jiang2017fast,liao2019unconditional}. 

{\em Failure of direct computation.} We first show that the roundoff problem can result in completely wrong results. 
We test the standard L2 scheme and the fast L2 scheme for different numbers of time steps, using directly the explicit formulas of $a_j^{(k)}$,  $\tilde c_j^{(k)}$ in \eqref{eq:akjexplicitit} and $
{a}^{k,\ell}_{k-1}$, $\tilde{c}^{k,\ell}_{k-1}$ in \eqref{eq:soeexplic}.
This is done by setting the thresholds $\theta_{s,1}=\theta_{s,2}=\theta_{f,1}=\theta_{f,2}=0$ in Algorithm~\ref{alg:akjckj}--\ref{alg:fastack}.
In Table \ref{tabwrongL2}, it is observed that the maximum error, i.e. 
$${\tt err}_{\rm max} \coloneqq \max_{1\leq k\leq N} \|u(t_k)-u^k\|$$
with $\|\cdot\|$ the $L^2$-norm in $\Omega$, can go up to $6.0231\text{e+7}$ for standard L2 scheme and to $4.6860$ for fast L2 scheme, implying that the direct computation is unreliable. 

\begin{table}[!]
\renewcommand\arraystretch{1.1}
\begin{center}
\def\temptablewidth{0.85\textwidth}
\caption{(Example~\ref{exm1})  Maximum $L^2$-errors ${\tt err}_{\rm max}$ of standard L2 scheme with $\alpha=0.4$ (top) and fast L2 scheme with $\alpha=0.5$ (bottom) for different $N$, where the coefficients are computed directly from explicit formulas and the grading parameter is $r=2/\alpha$. }\vspace{0in}\label{tabwrongL2}
{\rule{\temptablewidth}{1pt}}
\begin{tabular*}{\temptablewidth}{@{\extracolsep{\fill}}ccccc}
  $N=200 $ & $N=400 $&$N=800$&$N=1600$&$N=3200$\\ \hline
2.0980e-1&	7.9774e+1&	2.4502e+4&	5.4546e+6&	6.0231e+7\\
  \hline
 2.02887e-4&	1.2529e-1&	1.3868&	5.7361e-1&	4.6860\\
\end{tabular*}
{\rule{\temptablewidth}{1pt}}
\end{center}
\end{table}

{\em Feasibility of Taylor-expansion-based method.} 
Now, we use Algorithm~\ref{alg:akjckj}--\ref{alg:fastack} with the aforementioned default settings, to compute the coefficients $a_j^{(k)}$,  $\tilde c_j^{(k)}$, ${a}^{k,\ell}_{k-1}$, $\tilde{c}^{k,\ell}_{k-1}$. 
The maximum $L^2$-errors  for different $\alpha$  are showed  in Figure-\ref{fig:SartFast} for standard and fast L2-TCTE schemes.
It can be verified that the convergence rates of maximum errors are $\min\{r\alpha, 3-\alpha\}$, which is consistent with the global convergence result in \cite[Theorem 5.2]{kopteva2021error}.

We show the $L^2$-errors at final time $T=1$, i.e.
$${\tt err}_{T} \coloneqq \|u(T)-u^N\|$$
for standard L2-TCTE scheme with different $\alpha$ and $r$ in Figure \ref{fig:L2pterror}, using Algorithm~\ref{alg:akjckj}--\ref{alg:fastack}.
The convergence rate of final-time errors is observed to be $r$ when $r<3-\alpha$ and $3-\alpha$ when $r>3-\alpha$, which is consistent with the pointwise-in-time convergence result in \cite[Theorem 5.2]{kopteva2021error}.
\begin{figure}
    \centering
   \includegraphics[trim={2in 0in 1in 0in},clip,width=1\textwidth]{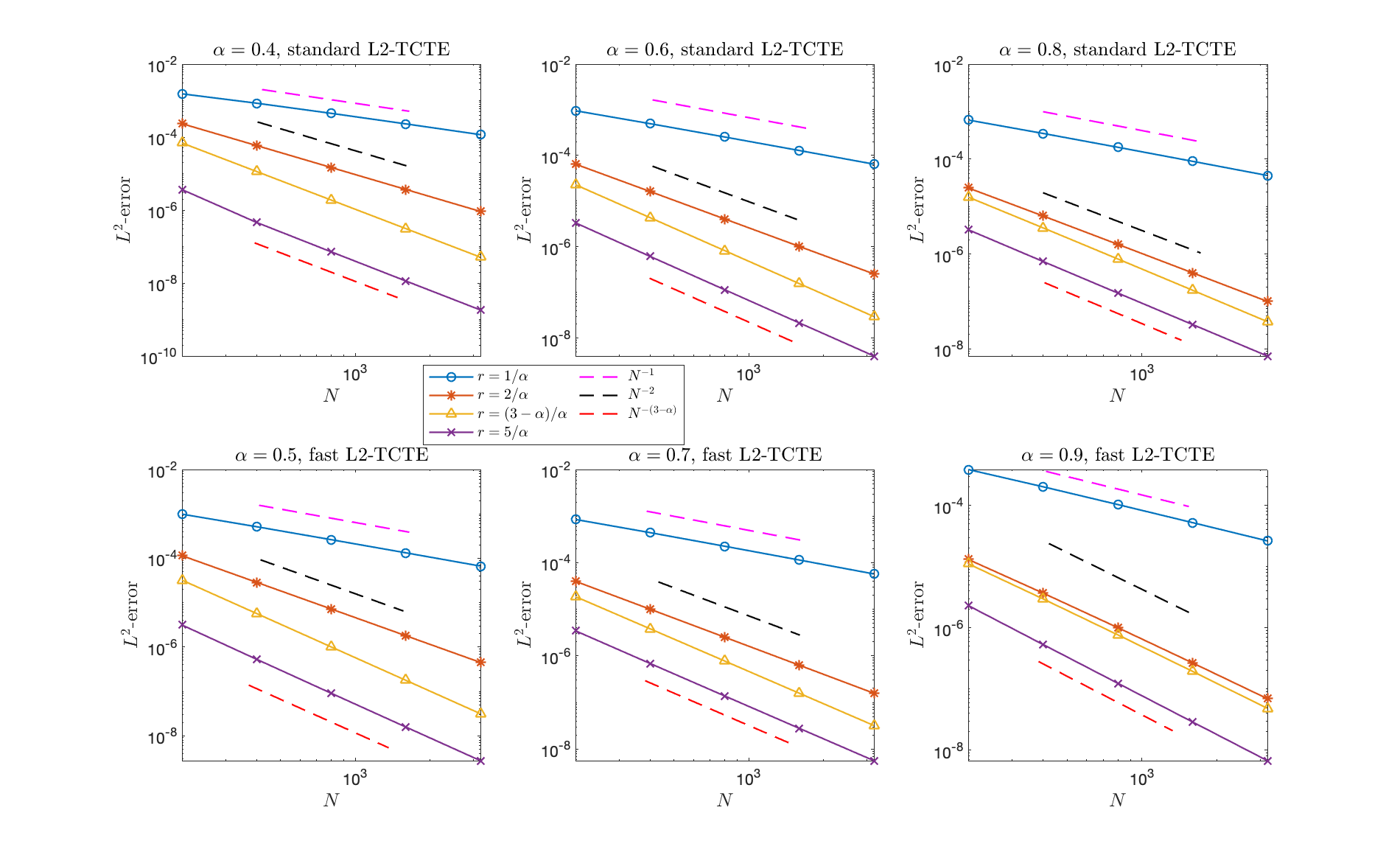}
   \vspace{-0.2in}    
    \caption{(Example~\ref{exm1}) Maximum $L^2$-errors ${\tt err}_{\rm max}$. Top:  Standard L2-TCTE method on the graded meshes with different $r$ and $N$ for $\alpha=0.4$, $\alpha=0.6$, $\alpha=0.8$ from left to right. Bottom: Fast L2-TCTE method on the graded meshes with different $r$ and $N$ for $\alpha=0.5$, $\alpha=0.7$, $\alpha=0.9$ from left to right.}
    \label{fig:SartFast}
\end{figure}
\begin{figure}
    \centering
   \includegraphics[trim={2in 0in 1in 0in},clip,width=1\textwidth]{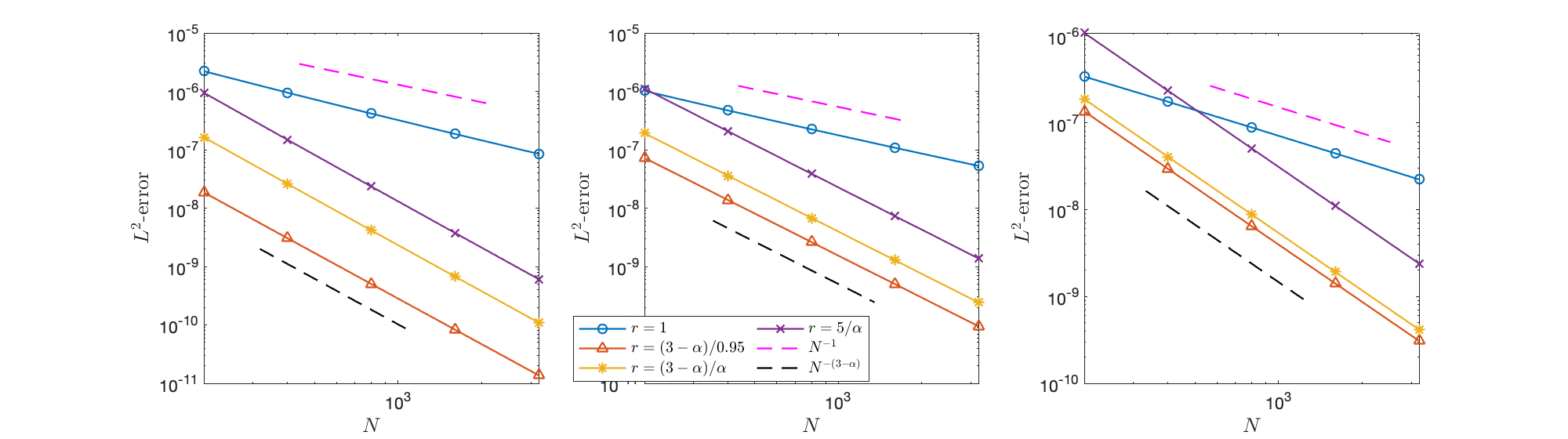}
   \vspace{-0.2in}    
    \caption{(Example~\ref{exm1}) Final-time $L^2$-errors ${\tt err}_{T}$ for the standard L2-TCTE method  on the graded meshes with different $r$ and $N$ for $\alpha=0.4$, $\alpha=0.6$, $\alpha=0.8$ from left to right.}
    \label{fig:L2pterror}
\end{figure}

Denote the $L^2$-error at time $t_k$ by
$$
{\tt err}_{t_k} \coloneqq \|u(t_k)-u^k\|. 
$$
We show the differences of pointwise $L^2$-errors, 
$$
{\tt err}_{t_k}^{\rm GK} -{\tt err}_{t_k}^{\rm TCTE},
$$
between using the standard L2-Gauss  and standard L2-TCTE  methods (the left-hand side of Figure \ref{fig:errfig}) and the fast L2-Gauss and fast L2-TCTE  methods (the right-hand side of Figure \ref{fig:errfig}).
It is observed that the differences are almost the machine error, which verifies the accuracy of our method. 
However, we shall mention that the Gauss--Kronrod quadrature method could be much more expensive than our TCTE method, which will be discussed in the next subsection. 
\begin{figure}
    \centering
    \includegraphics[width=1\textwidth]{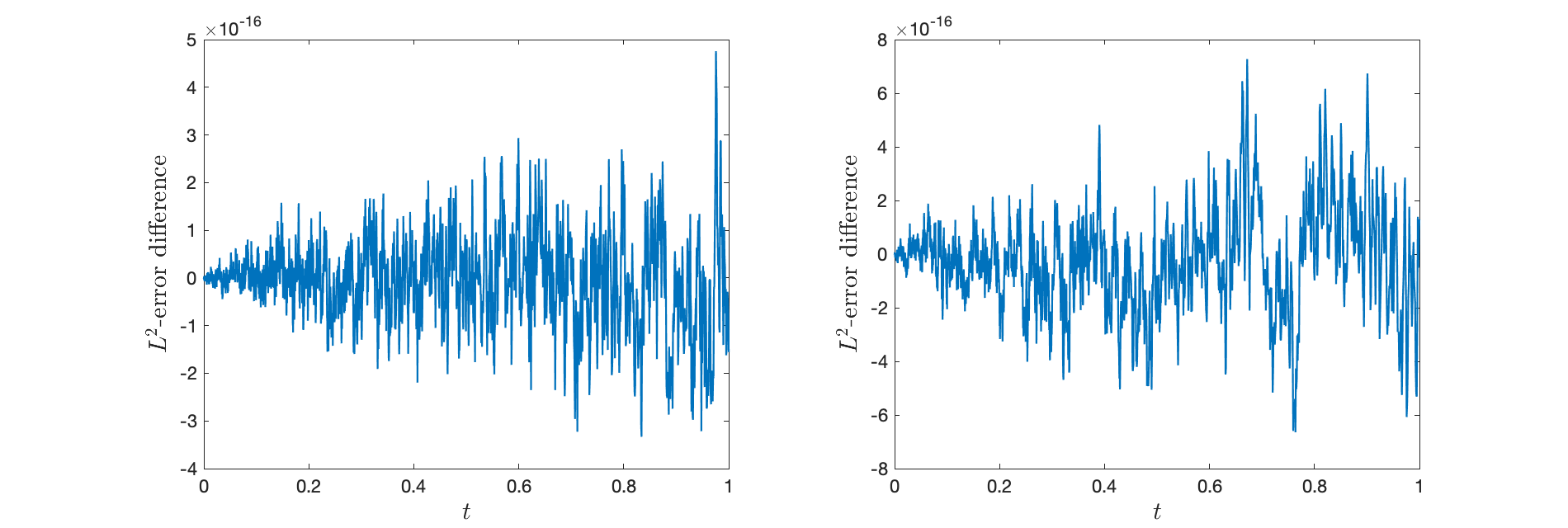}
   \vspace{-0.2in}    
    \caption{(Example~\ref{exm1}) Pointwise $L^2$-error differences between the Gauss--Kronrod quadrature method and the TCTE method  w.r.t. $t$, where $T=1$ and $N=3200$. Left: Standard L2 method with $\alpha=0.6$ and $r=(3-\alpha)/0.95$. Right: Fast L2 method with $\alpha=0.7$ and $r=(3-\alpha)/\alpha$.}
    \label{fig:errfig}
\end{figure}

\subsection{High efficiency of TCTE method}
\begin{example}\label{exm2}
  Consider the subdiffusion equation
  \eqref{eq:subdiffusion} in $\Omega=[-1,1]^2$ with $T=10$ and
 $f(t,x,y)=\Gamma(1+\alpha)(x^2-1)(y^2-1)-2t^\alpha(x^2+y^2-2)$, whose exact solution is $u(t,x,y)=t^\alpha(x^2-1)(y^2-1)$.
\end{example}
In this example, we set the final time of SOE $T_{\rm soe}=T$, the SOE tolerance $\varepsilon=1\text{e-12}$, the cut-off time $\Delta t= \tau_2$ in SOE approximations; see \cite{jiang2017fast,liao2019unconditional}.  
Since the  exact solution is polynomial of degree 2 in $x$ and $y$ direction, we  use the spectral collocation method in space with $5^2$ Chebyshev--Gauss--Lobatto points so that the error caused by space discretization is negligible.
  
Maximum error and $L^2$-error at final time $T=10$  and  CPU time (in seconds) for the  standard L2-Gauss, standard L2-TCTE, fast L2-Gauss and fast L2-TCTE  methods are given in Table~\ref{GaussTaylortimecom}, where  $\alpha=0.6$ and $r=(3-\alpha)/\alpha$.  
We observe that the CPU time of the fast L2 scheme increases approximately linearly w.r.t. $N$, while the CPU time of the standard L2 scheme increases approximately quadratically.   
Optimal convergence rate $3-\alpha$ is achieved, which is consistent with the convergence result \cite[Theorem 5.2]{kopteva2021error}. 
For a fixed $N$,  the errors of the standard L2-Gauss and L2-TCTE methods are almost the same in sense that their differences are approximately the machine error. 
Influenced by the chosen SOE tolerance error $\varepsilon=1\text{e-12}$, the differences of $L^2$-errors between the standard and fast methods are about  $\varepsilon$. 

As one can see, when $N=32000$, the standard L2-TCTE method about 91 times faster than the standard L2-Gauss method, while the fast L2-TCTE method is about 308 times faster than the fast L2-Gauss method in this example.

{\footnotesize
\begin{table}[thp]
\renewcommand\arraystretch{1.1}
\begin{center}
\def\temptablewidth{1\textwidth}
\caption{(Example~\ref{exm2}) Maximum errors, final-time errors, and CPU times (in seconds) computed by standard L2-Gauss, standard L2-TCTE, fast L2-Gauss and fast L2-TCTE methods with $\alpha=0.6$,  $r=(3-\alpha)/\alpha$ and $T=10$ for different numbers of time steps.}\vspace{0in}\label{GaussTaylortimecom}
{\rule{\temptablewidth}{1pt}}
\begin{tabular*}{\temptablewidth}{@{\extracolsep{\fill}}ccccccc}
 & & $N=2000$&$N=4000 $ & $N=8000 $&$N=16000$&$N=32000$\\ \hline
standard L2-Gauss &    $ {\tt err}_{\rm max}$&3.8628e-7&	7.3187e-8&	1.3866e-8&	2.6271e-9&	4.9776e-10\\
& & -- & 2.4000    &2.4000  &  2.4000  &  2.4000\\
 &    ${\tt err}_{T}$&3.1934e-9&	6.0409e-10	&1.1434e-10	&2.1654e-11&	4.0991e-12\\
& & -- &   2.4023  &  2.4013  &  2.4007  &  2.4013\\
  \cline{2-7}
 &CPU time&9.5653\text{e+2}	&3.8583\text{e+3}	&1.5218\text{e+4}	&	6.0095\text{e+4}	&	2.4060\text{e+5}	\\
 \cline{1-7}
  standard  L2-TCTE &    $ {\tt err}_{\rm max}$  &3.8628e-7	&7.3187e-8&	1.3866e-8&	2.6271e-9&	4.9776e-10\\
 && --  &   2.4000    &2.4000  &  2.4000  &  2.4000\\
  &    ${\tt err}_{T}$ &3.1934e-9 &	6.0409e-10 &	1.1434e-10 &	2.1654e-11 &	4.0998e-12\\
   && --  &  2.4023&     2.4013   &  2.4007 &    2.4010\\
  \cline{2-7}
 &CPU time&8.5312&   35.3594 & 147.7656 & 623.7969&	2639.0823
\\
\cline{1-7}
fast L2-Gauss    &    $ {\tt err}_{\rm max}$ &3.8628e-7&	7.3187e-8&	1.3866e-8	&2.6271e-9	&4.9776e-10\\
 && -- &  2.4000  &   2.4000 &    2.4000   &  2.4000\\
  &    ${\tt err}_{T}$ &3.1935e-9&	6.0417e-10&	1.1442e-10&	2.1731e-11&	4.1806e-12\\
& & -- &   2.4021  &  2.4005    & 2.3966  &   2.3779
\\   \cline{2-7}
 &CPU time&134.0000&	307.0156&	634.1875&	1322.6875&	2810.9687
\\\cline{1-7}
fast L2-TCTE &    $ {\tt err}_{\rm max}$ & 3.8628e-7&	7.3187e-8&	1.3866e-8&	2.6271e-9&	4.9776e-10\\
 && -- &  2.4000  &   2.4000 &    2.4000   &  2.4000\\
 &    ${\tt err}_{T}$ &3.1935e-9&	6.0417e-10&	1.1442e-10&	2.1730e-11&	4.1803e-12\\
& & -- &  2.4021 &   2.4005  &  2.3967    &2.3780
\\   \cline{2-7}
 &CPU time&   0.4375 &  0.9062 &   1.9688 &   4.2500 &   9.1250
\\

\end{tabular*}
{\rule{\temptablewidth}{1pt}}
\end{center}
\end{table}}

\subsection{Long time simulation}
For simplicity, we still consider Example \ref{exm2} but replacing the final time $T=10$ with $T=1000$.
We set the final time of SOE $T_{\rm soe}=T$, the SOE tolerance $\varepsilon=1\text{e-14}$, the SOE cut-off time $\Delta t= \tau_2$, and the grading parameter $r=(3-\alpha)/\alpha$. 
Still,  $5^2$ Chebyshev--Gauss--Lobatto points are used in the spectral collocation method for spatial discretization. 
The fast L2-TCTE method is adopted for this long time simulation.

In Table \ref{tab1FastLo}, the maximum $L^2$-errors, the final-time $L^2$-errors, the convergence rates, and the CPU times are reported. 
We can observe that the convergence rates are optimal and what's more important, the CUP time is not large even for very large number of time steps ($N$ is larger than one hundred of thousands). 
To the best of knowledge, there are no existing simulations in literatures with more than $10^6$ time steps. 
The high efficiency of our TCTE method helps to achieve this.

\begin{table}[!]
\renewcommand\arraystretch{1.1}
\begin{center}
\def\temptablewidth{1\textwidth}
\caption{(Example~\ref{exm2}) Maximum errors, final-time errors, and CPU times (in seconds) computed by fast L2-TCTE method with $r=(3-\alpha)/\alpha$ and $T=1000$ for different numbers of time steps.}\vspace{0in}\label{tab1FastLo}
{\rule{\temptablewidth}{1pt}}
\begin{tabular*}{\temptablewidth}{@{\extracolsep{\fill}}ccccccc}
 & & $N=8000$&$N=16000 $ & $N=32000 $&$N=64000$&$N=128000$\\ \hline
$\alpha=0.4$ &    $ {\tt err}_{\rm max}$&7.9773e-8&	1.3157e-8&	2.1702e-9&	3.5795e-10&	5.9040e-11\\
& & -- &  2.6000  &  2.6000   & 2.6000  &  2.6000\\
 &    ${\tt err}_{T}$&4.3198e-11&	7.0758e-12&	1.1718e-12&	1.9549e-13&	4.2172e-14\\
& & -- &   2.6100 &    2.5941 &    2.5836  &   2.2127\\
  \cline{2-7}
 &CPU time& 4.3906   & 6.8438&   13.5156   &27.9844 &  58.1406
\\
\cline{1-7}
$\alpha=0.6$ &    $ {\tt err}_{\rm max}$  &2.1976e-7&	4.1638e-8&	7.8889e-9&	1.4946e-9&	2.8318e-10\\
 && --  &   2.4000    &2.4000  &  2.4000  &  2.4000\\
  &    ${\tt err}_{T}$ &1.1247e-10&	2.1314e-11&	4.0756e-12	&7.6765e-13&	1.3773e-13\\
 && --  &  2.3998   & 2.3867 &   2.4085  &  2.4785\\
  \cline{2-7}
 &CPU time&  2.2031 &   4.8281 &  10.1094  & 21.5000 &  46.1406
\\
\cline{1-7}
$\alpha=0.8$  &    $ {\tt err}_{\rm max}$ &1.3302e-6	&2.9038e-7	&6.3250e-8&	1.3768e-8&	2.9966e-9\\
 && -- & 2.1956   & 2.1988 &   2.1997   & 2.1999\\
  &    ${\tt err}_{T}$ &2.2127e-10&	4.8295e-11&	1.0636e-11&	2.4319e-12&	5.7001e-13\\
& & -- & 2.1959 &    2.1829&     2.1288 &    2.0930
\\   \cline{2-7}
 &CPU time&2.0000  &  4.0312  &  8.9219 &  18.7812 &  40.8125
\\
\end{tabular*}
{\rule{\temptablewidth}{1pt}}
\end{center}
\end{table}

\section{Conclusion}\label{sect5}
In this work, to handle the roundoff error problems in L2-type method, we reformulate the standard and fast L2 coefficients, provide the threshold conditions on $\delta$-cancellation, and propose the Taylor expansions to avoid $\delta$-cancellation when the threshold conditions are not satisfied. 
Numerical experiments show the high accuracy and efficiency of our proposed TCTE method.
In particular, the fast L2-TCTE method can complete long time simulations with hundreds of thousands of time steps in one minute. 

\section*{Acknowledgements}
C. Quan is supported by NSFC Grant 12271241, the fund of the Guangdong Provincial Key Laboratory of Computational Science and Material Design (No. 2019B030301001), and the Shenzhen Science and Technology Program (Grant No. RCYX20210609104358076).

\bibliography{bibfile}
\bibliographystyle{plain}

\end{document}